\documentclass{amsart}
\usepackage{amssymb}
\usepackage{amscd,enumerate,amsfonts,calc,amsmath,verbatim,xypic}

\newcommand{\Ext}{\operatorname{Ext}}
\newcommand{\Tor}{\operatorname{Tor}}
\newcommand{\Soc}{\operatorname{Soc}}
\newcommand{\Hom}{\operatorname{Hom}}
\newcommand{\Hilb}{\operatorname{Hilb}}
\newcommand{\HH}{\operatorname{H}}

\newcommand{\Coker}{\operatorname{Coker}}

\newcommand{\rank}{\operatorname{rank}}

\newcommand{\type}{\operatorname{type}}

\newcommand{\bd}{\boldsymbol}
\newcommand{\ov}{\overline}
\newcommand{\into}{\hookrightarrow}
\newcommand{\col}{\colon}
\newcommand{\xra}{\xrightarrow}
\newcommand{\lra}{\longrightarrow}
\newcommand{\dd}{\partial}

\newcommand{\m}{{\mathfrak m}}
\newcommand{\fm}{{\mathfrak m}}

\newcommand{\bw}{{\mathsf\Lambda}}

\theoremstyle{remark}

\swapnumbers

\theoremstyle{plain}
\newtheorem{theorem}{Theorem}[section]
\newtheorem{conjecture}[theorem]{Conjecture}
\newtheorem*{conjectureAR}{Conjecture (Auslander-Reiten)}

\newtheorem*{theoremAR}{{\rm \ref{AR}.}\;Theorem}
\newtheorem*{theoremm^2M}{{\rm \ref{m^2M}.}\;Theorem}
\newtheorem*{conj}{{\rm \ref{conjecture}.}\;Conjecture}

\newtheorem*{theoremomega}{{\rm \ref{omega}.}\;Theorem}

\newtheorem*{conjectureT}{Conjecture (Tachikawa)}

\newtheorem{proposition}[theorem]{Proposition}
\newtheorem{lemma}[theorem]{Lemma}

\theoremstyle{definition}
\newtheorem{chunk}[theorem]{}
\newtheorem{subchunk}{}
\newtheorem*{chunk*}{}
\newtheorem{example}[theorem]{Example}

\theoremstyle{remark}

\newtheorem*{Remark}{Remark}
\newtheorem{remark}[theorem]{Remark}

\numberwithin{equation}{theorem}

\numberwithin{subchunk}{theorem}

\begin{document} \title[Vanishing of Ext and Tor
]{Vanishing of Ext and Tor \\over Cohen-Macaulay
local rings}
\author[C.~Huneke]{Craig Huneke}
\address{Mathematics Department, University of Kansas, Lawrence, KS 66045}
\email{huneke@math.ukans.edu}

\author[L.~M.~\c Sega]{Liana M.~\c Sega}
\address{Mathematical Sciences Research Institute, 1000 Centennial Drive,
Berkeley, CA 94720}
\email{lsega@msri.org}

\author[A.~N.~Vraciu]{Adela N.~Vraciu}
\address{Mathematics Department, University of Kansas, Lawrence, KS 66045}
\email{avraciu@math.ukans.edu}
\date{\today} \maketitle

\section*{Introduction}

In recent years, there has been growing interest in understanding the
vanishing properties of Ext and Tor over Noetherian local rings,
especially in the case when the ring is Artinian. One motivation for
this interest is given by a conjecture of Auslander and Reiten
\cite{AR}, which in the case of commutative local rings, can be
stated as follows:

\begin{conjectureAR}Let $(R,\m)$ be a commutative Noetherian local
ring, and $M$ a finitely generated $R$-module. If $\Ext^i_R(M,M\oplus
R) = 0$ for all $i > 0$, then $M$ is free.
\end{conjectureAR}

This conjecture was intially stated for Artin algebras, and Auslander,
Ding and S\o lberg \cite{ADS} widened the context to algebras over
commutative local rings. A recent result of Huneke and Leuschke
\cite{HL} establishes the conjecture in the case when $R$ is an
excellent Cohen-Macaulay normal domain containing the rational
numbers.

To prove the Auslander-Reiten Conjecture for Cohen-Macaulay rings, it
suffices to consider the case of Artinian rings. Indeed, if $R$ is
assumed to be Cohen-Macaulay, then one can first replace $M$ by a high
syzygy in a minimal free resolution of $M$ (see \cite{AR}) to assume
that $M$ is maximal Cohen-Macaulay. If $x_1,...,x_d$ is a maximal $M$-
and $R$-regular sequence, and $I$ is the ideal generated by it, then
replacing $R$ by $R/I$ and $M$ by $M/IM$, one can assume without loss
of generality that $R$ is Artinian.

In this paper we chiefly concentrate on the commutative Artinian
case. If $\m^2 = 0$, then the first syzygy in a minimal free
resolution of any non-free $R$-module is annihilated by the maximal
ideal, and the Auslander-Reiten conjecture follows trivially.  The
first interesting open case is when $\m^3 = 0$.

Rings in which $\m^3 = 0$ were systematically studied by Lescot
\cite{Le}. In particular, his results give the Poincar\'e series of
finitely generated modules none of whose minimal syzygies split off a
copy of the residue field. Only such modules could provide
counterexamples to the Auslander-Reiten conjecture. For if
$\Ext_R^i(M,M\oplus R) = 0$ for all $i$, and $M$ is not free, it is
clear by shifting degree that no syzygy of $M$ can have the residue
field as a direct summand.

One of our main results proves the Auslander-Reiten conjecture for
rings with $\fm^3=0$. Note that the statement gives an effective bound
on the required number of vanishing Ext modules.

\begin{theoremAR}
Let $(R,\fm)$ be a commutative Artinian local ring with $\fm^3=0$ and $M$
a
finitely generated $R$-module.
\begin{enumerate}[\quad\rm(1)]
\item If $\Ext^i_R(M, M\oplus R)=0$ for four consecutive
values of $i$ with $i\ge 2$, then $M$ is free.
\item If $R$ is Gorenstein and $\Ext^i_R(M,M)=0$ for some $i>0$, then
      $M$ is free.
\end{enumerate}
\end{theoremAR}

The second part extends a result of Hoshino
\cite[?]{Hoshino} on (possibly non-commutative) finite dimensional
self-injective algebras with radical cube zero.

As mentioned above, in the context of the Auslander-Reiten conjecture
one can replace the original module $M$ by an arbitrary minimal syzygy
of $M$, and it will satisfy the same vanishing properties. Now, if
$\m^3 = 0$ and $N$ is a first syzygy of any $R$-module, then $\m^2N =
0$. A closer look at modules annihilated by $\m^2$ shows that the
Auslander-Reiten conjecture holds for any such module over an
arbitrary Artinian local ring. We also give a bound on the required
number of vanishing Exts, in terms of the
minimal number of generators, denoted $\nu(-)$, of certain modules:

\begin{theoremm^2M}
Let $(R,\fm)$ be a commutative Artinian local ring and $M$ a finitely
generated $R$-module
with $\fm^2M=0$.

If  $\Ext^i_R(M, M\oplus R)=0$ for all $i$ with
  $0<i\le \max \lbrace 3, \nu (M), \nu (\fm M)\rbrace $, then $M$ is
  free.
\end{theoremm^2M}

Our results are proved by considering more generally the vanishing of
$\Ext_R^*(M,N)$ where $M$ and $N$ are two finitely generated modules
over an Artinian ring $R$. Since the Matlis dual of such Ext modules
are Tor modules, we often find it more convenient to work with the
vanishing properties of Tor. Our arguments suggest that the vanishing
of $\Tor_i^R(M,N)$ for all positive $i$ places restrictions which
relate the annihilators of $M$, $N$ and $R$. Specifically, we propose
the following:

\begin{conj}
  Let $R$ be a commutative Artinian local ring and let $M$, $N$ be
  nonzero modules with $\fm^2M=\fm^2N=0$. If $\Tor_i^R(M,N)=0$ for all
  $i>0$, then $\fm^3=0$.
\end{conj}

One can ask a more general question: Let $p$, $q$ be positive
integers and assume that $M$, $N$ are nonzero modules with
$\fm^pM=\fm^qN=0$. If $\Tor_i^R(M,N)=0$ for all $i>0$, does it follow
that $\fm^{p+q-1}=0$?

The hypothesis of Conjecture \ref{conjecture} imposes a strong
condition on the Poincar\'e series of the residue field of $R$,
cf. Lemma \ref{Poincare}. This shows that the conclusion holds for a
large class of rings, including complete intersections of codimension
greater than $2$, Koszul rings, Golod rings, etc. In Theorem
\ref{graded} we prove the conjecture when the ring $R$ is standard
graded.

Another conjecture which has received attention recently is a
conjecture of Tachikawa. A commutative version of this conjecture for
Cohen-Macaulay local rings is the following, cf. Avramov, Buchweitz
and \c Sega \cite{ABS}, and also Hanes and Huneke \cite{HH}:

\begin{conjectureT}
Let $(R,\m)$ be a Cohen-Macaulay local ring.
If $R$ has
a canonical module $\omega$ and $\Ext^i_R(\omega,R) = 0$ for all
$i > 0$, then $R$ is Gorenstein, i.e. $\omega$ is free.
\end{conjectureT}

This version of the Tachikawa conjecture is subsumed by the
Auslander-Reiten conjecture, since the condition
$\Ext^i_R(\omega,\omega) = 0$ for all $i>0$ is automatic when $R$ is
Cohen-Macaulay.

Assuming that $\fm^3=0$ and $R$ is a finite dimensional algebra over a
field, this conjecture was proved by Asashiba \cite{As} under the weaker
assumption that $\Ext^1_R(\omega,R)=0$.

Recall that $\Ext^i_R(\omega,R)$ is Matlis dual to $\Tor_i^R(\omega,
\omega)$ when $R$ is Artinian. Thus, the theorem below is equivalent
to the one that appears in Section \ref{Rings with m^3=0} under the
same number.

\begin{theoremomega}
  Let $(R, \fm )$ be a commutative Artinian local ring with $\fm ^3=0$.
The
  following statements are equivalent:
\begin{enumerate}[\quad\rm(1)]
\item $R$ is Gorenstein.
\item $\Ext^1_R(\omega, R)=0$.
\item $\Ext^2_R(\omega, R)=\Ext^3_R(\omega, R)=0$.
\item $\Ext^j_R(\omega, R)=\Ext^{j+1}_R(\omega,
      R)=\Ext^{j+2}_R(\omega, R)=0$ for some $j \ge 3$.
\end{enumerate}
\end{theoremomega}

A different proof of the equivalence $(1)\iff (2)$, along the lines of
\cite{AsHo}, is given in \cite{ABS}.

Section \ref{Betti numbers} contains a number of lemmas we will use
throughout the paper, concerning the growth of Betti number of modules
under certain vanishing of Tor conditions.

Section \ref{Rings with m^3=0} contains our work on rings with $\m^3 =
0$. An important technical result is Theorem \ref{3tors}, which gives
detailed information comparing the Betti numbers of two modules with
three consecutive vanishing Tors. In this section we prove Theorem
\ref{omega}.

In Section \ref{Rings with large embedding dimension} we show that if
enough Tors vanish for two modules $M$,$N$ and $\fm^2M=0$, then either
$M$ or $N$ is free. However, for this result we impose strident
conditions on the ring.

Section \ref{The Auslander-Reiten conjecture} contains the proof of
Theorems \ref{AR} and \ref{m^2M}.

In Section \ref{Vanishing of Tor and the Loewy length} we deal with
Conjecture \ref{conjecture}, and give the proof of Theorem
\ref{graded}.

In Section \ref{Rings of type at most 2} we prove Tachikawa's
conjecture for Cohen-Macaulay rings of type two.

\section{Betti numbers}
\label{Betti numbers}

In this paper $(R,\fm,k)$ denotes a commutative Noetherian local ring
with maximal ideal $\fm$ and residue field $k$. We consider finitely
generated $R$-modules $M$, $N$.

The number $\nu(M)$ denotes the minimal number of generators of $M$ and
$\lambda(M)$ denotes the length of $M$.

For every $i\ge 0$ we let $M_i$ denote the $i$-th syzygy of $M$ in a
minimal free resolution \[ \dots\lra R^{b_{i+1}(M)}\xrightarrow{\
\delta_i\ } R^{b_{i}(M)}\lra \dots\xra{\ \delta_0\ } R^{b_0(M)}\lra
M\lra 0 \]

The number $b_i(M)$ is called the $i$-th {\em Betti number} of $M$
over $R$. The {\em Poincar\'e series} of $M$ over $R$ is the formal
power series
\[
P^R_M(t)=\sum_{i=0}^{\infty}b_i(M)t^i
\]

In this section we describe several restrictions on the Betti numbers
of $M$,$N$ that are imposed under the assumption that
$\Tor^R_i(M,N)=0$ for certain values of $i$.

\medskip

\begin{remark}
 \label{summands}
 If $\Tor^R_i(M,N)=0$ for some $i>0$ and $N$ has infinite projective
 dimension, then $k$ is not a direct summand of any of the $R$-modules
 $M_0,\dots, M_{i-1}$.  Indeed, we have
 $\Tor^R_{i-j}(M_j,N)\cong\Tor^R_{i}(M,N)=0$ for all $j<i$. If $k$ is
 a direct summand of $M_j$ for some $j<i$, then $\Tor^R_{i-j}(k,N)=0$,
 contradicting the assumption on $N$.
 \end{remark}

Let $P(t)=a_0+a_1t+\dots+a_it^i+\dots$ be a formal power series. For
each $n\ge 0$ we denote $[P(t)]_{\le n}$ the polynomial
$a_0+a_1t+\dots+a_nt^n$.

The next result is a slightly modified version of a technique in
\cite[1.1]{CM}:
\begin{lemma}
\label{product}
Let $n$ be a positive integer. If $\Tor_i^R(M,N)=0$ for all $i\in
[1,n]$, then
\[
\big [ P^R_{M\otimes_RN}(t)\big]_{\le n}=\big[P^R_M(t)P^R_N(t)\big
]_{\le n}
\]
\end{lemma}

\noindent{\it Proof.} Let $F$, respectively $G$, be a minimal free
resolution of $M$, respectively $N$, over $R$. The hypothesis implies
that $\HH_i(F\otimes_RG)=0$ for all $i\in[1,n]$. Since $F\otimes_RG$
is a minimal complex with $\HH_0(F\otimes_RG)=M\otimes_RN$, we note
that $(F\otimes_RG)_{\le n}$ is the beginning of a minimal free
resolution of $M\otimes_RN$. We have thus:
\begin{xalignat*}{3}
& &   [P^R_{M\otimes_RN}&]_{\le n}=\sum_{i=0}^n
\rank(F\otimes_RG)_it^i=[P_M^R(t)P_N^R(t)]_{\le n}  & & \qed
\end{xalignat*}

\begin{chunk} For each nonzero $R$-module $M$ of finite length we set
\[
\gamma(M)=\frac{\lambda(M)}{\nu(M)}-1
\]

Note that $\gamma(M)$ is also equal to $\displaystyle{\frac{\lambda
    (\m M)}{\nu(M)}}$. It is thus a rational number in the interval
$[0,\lambda(R)-1]$. The extreme values on this interval are attained
as follows: $\gamma(M)=0$ if and only if $\fm M=0$ and $\gamma(M)=\lambda
(R)-1$ if and only if $M$ is free.

We will often use the definition of $\gamma(M)$ in length
computations as follows:
\begin{equation}
\label{length}
\lambda(M)=\nu(M)\big(\gamma(M)+1\big)
\end{equation}
\end{chunk}

\begin{lemma}
\label{gammas}
Let $R$ be an Artinian ring and let $M$, $N$ be finiteley generated
$R$-modules such that $M$ is not zero and $N$ is not free.
\begin{enumerate}[\quad\rm (1)]
\item If $\Tor_i^R(M,N)=0$ for some $i>0$, then
\[
\big(\gamma({M\otimes_RN_i})+1\big)b_i(N)=\big(\gamma(M)-\gamma({M\otimes_RN_{i-1}})\big)b_{i-1}(N)
\]
In particular, there is an inequality $b_{i}(N)\le \gamma(M)
b_{i-1}(N)$.

\item If $\fm ^2M=0$ and $\Tor_i^R(M,N)=0$ for some $i>0$, then
      $\fm(M\otimes_RN_i)=0$ and
\[
b_i(N)=\big( \gamma (N)-\gamma (M\otimes_R N_{i-1})\big) b_{i-1}(N)
\]

\item  If $\fm ^2M=0$ and $\Tor_i^R(M,N)=\Tor_{i-1}^R(M,N)=0$ for some
$i>1$, then
\[
b_i(N)=\gamma (M) b_{i-1}(N)
\]
\end{enumerate}
\end{lemma}

\begin{proof}
 (1) Consider the short exact sequence
\[
0\lra N_i\lra R^{b_{i-1}(N)}\lra N_{i-1}\lra 0
\]
The hypothesis implies that the sequence remains exact when tensored with
$M$:
\begin{equation}
\label{ses}
0\lra M\otimes_RN_i\lra M\otimes_RR^{b_{i-1}(N)}\lra M\otimes_RN_{i-1}\lra
0
\end{equation}

For any $j$ we use (\ref{length}) to obtain
\begin{gather*}
\lambda(M\otimes_RN_j)=\nu(M\otimes_R
N_j)\big(\gamma(M\otimes_RN_j)+1\big)=\nu(M)b_j(N)\big(\gamma(M\otimes_RN_j)+1\big)\\
\lambda(M\otimes_RR^{b_j(N)})=b_j(N)\lambda(M)=b_j(N)\nu(M)\big(\gamma(M)+1\big)
\end{gather*}

Using these expressions, a length count in (\ref{ses}) leads to the
desired conclusion.

(2) As above, we have a short exact sequence (\ref{ses}). The image of
$N_i$ in $R^{b_{i-1}(N)}$ is contained in $\fm R^{b_{i-1}(N)}$, hence
the image of $M\otimes_RN_i$ in $M\otimes_RR^{b_{i-1}(N)}$ is
contained in $\fm (M\otimes_RR^{b_{i-1}(N)})$, and the latter is
annihilated by $\fm$. We have then $\gamma({M\otimes_RN_i})=0$ and the
relation follows from (1).

(3) By (2) we have $\fm (M\otimes _R N_{i-1})=\fm (M\otimes_R
      N_i)=0$ and therefore $\gamma (M\otimes_R N_{i-1})=\gamma
      (M\otimes_R N_i)=0$.
\end{proof}

\begin{lemma}
\label{one}
Let $R$ be an Artinian ring and let $M$, $N$ be finitely generated
$R$-modules such that $M$ is not zero and $N$ is not free.
\begin{enumerate}[\quad\rm (1)]
\item If $\Tor_i^R(M,N)=0$ for all $i\in [1,\nu(N)]$, then
$\gamma(M)\ge 1$.

\item If $\fm^2M=0$ and $\Tor_i^R(M,N)=0$ for all $i\in [1, \log
   _2 b_1(N)+2]$, then $\gamma(M)$ is an integer.
\end{enumerate}
\end{lemma}

\begin{proof}
  (1) Assume that $\gamma(M)<1$. By Lemma \ref{gammas}(1) we have then
      \[
b_i(N)\le\gamma(M) b_{i-1}(N)<b_{i-1}(N)\quad\text{for all}\quad
i\in[1,\nu(N)]
\]
 Since
  $b_0(N)=\nu(N)$ we conclude that $b_{\nu(N)}(N)=0$, hence $N$ has
  finite projective dimension and it is thus free, contradicting the
  hypothesis.

 (2) Using Lemma \ref{gammas}(3) we have:
\[
b_{i+1}(N)=(\gamma(M))^ib_1(N)\quad\text{ for all}\quad i\in [1,\log _2
b_1(N)+1]
\]

Let $u,v$ be relatively prime positive integers such that
$\gamma(M)=uv^{-1}$. It follows that $v^i$ divides $b_1(N)$ for all
$i\in [1,\log _2 b_1(N)+1]$. If $v\ge 2$, then $b_1(N)\ge
2^i$ for all such $i$, a contradiction.
\end{proof}

Let $e$ denote the minimal number of generators of $\fm$.

\begin{lemma}
\label{relations}
Let $R$ be an Artinian ring, and let $M$, $N$ be non-free finitely
generated $R$-modules. If $\fm^2M=0$ and
$\Tor^R_2(M,N)=\Tor^R_1(M,N)=0$, then the following hold:

\begin{enumerate}[\quad\rm(1)]
\item $b_1(M)=\big(e-\gamma(M)\big)b_0(M)$.
\item $\fm M_1=\fm^2 R^{b_0(M)}$.
\end{enumerate}
\end{lemma}

\begin{proof}
  (1) Lemma \ref{gammas}(2) shows that the $R$-module $M\otimes_RN_1$
      is a finite direct sum of copies of $k$, hence its first Betti
number
      is $eb_0(M)b_1(N)$. On the other
  hand, since $\Tor_1(M,N_1)=0$, Lemma \ref{product} gives
  $b_1(M\otimes_RN_1)=b_0(M)b_2(N)+b_1(M)b_1(N)$. We also have
  $b_2(N)=\gamma(M) b_1(N)$ by Lemma \ref{gammas}(3), hence
$$eb_0(M)b_1(N)=b_0(M)\gamma(M) b_1(N)+b_1(M)b_1(N)$$
(2) A length count in the short exact sequence
\[
0\lra M_1\lra R^{b_{0}(M)}\lra M\lra 0
\]
using (\ref{length}) gives
\begin{align*}
\lambda(M_1)&=\lambda(R)b_0(M)-\lambda(M)\\
&=\big(1+e+\lambda(\fm^2)\big)b_0(M)-b_0(M)\big(\gamma(M)+1\big)\\
&=\big(e+\lambda(\fm^2)-\gamma(M)\big) b_0(M)
\end{align*}
We next use (1) to obtain
\begin{align*}
\lambda(\fm M_1)&=\lambda(M_1)-\nu(M_1)\\
                &=\lambda(M_1)-b_1(M)\\
&=\big(e+\lambda(\fm^2)-\gamma(M)\big)b_0(M)-\big(e-\gamma(M)\big)b_0(M)\\
                &=\lambda(\fm^2)b_0(M)
\end{align*}

Since $\fm M_1$ is contained in $\fm^2R^{b_0(M)}$ and both modules have
the same length, it follows that they are equal.
\end{proof}

\begin{lemma}
\label{sum}
Let $R$ be an Artinian ring and let $M$, $N$ be non-free finite
$R$-modules with $\fm^2M=\fm^2N=0$.

If $\Tor^R_2(M,N)=\Tor^R_1(M,N)=0$, then
$\gamma(M)+\gamma(N)-\gamma({M\otimes_RN})=e$
 \end{lemma}

\begin{proof}  Compare the relations
$$
b_1(M)=\big(\gamma(N)-\gamma({M\otimes_RN})\big)b_0(M)\quad
\text{and}\quad
b_1(M)=\big(e-\gamma(M)\big)b_0(M)
$$
given by Lemma \ref{gammas}(2), respectively Lemma
\ref{relations}(1).
\end{proof}

\begin{lemma}
\label{Poincare}
Let $R$ be an Artinian ring and let $M$, $N$ be finitely generated
non-free
$R$-modules with $\fm^2M=\fm^2N=0$.

If $\Tor_i^R(M,N)=0$ for all $i>0$, then
\[
P_k^R(t)=\frac{1-\gamma({M\otimes_RN})t}{\big(1-\gamma(M)t\big)\big(1-\gamma(N)t\big)}
\]
\end{lemma}

\begin{proof}
By Lemma \ref{gammas} we have
\[
P_{N_1}^R(t)=\frac{b_1(N)}{\big(1-\gamma(M)t\big)}\quad\text{and}\quad
P_M^R(t)=b_0(M)+\frac{b_0(M)\big(\gamma(N)-\gamma({M\otimes_RN})\big)t}{\big(1-\gamma(N)t\big)}
\]
 Lemma \ref{product} then yields
\[
P^R_{M\otimes_RN_1}(t)=P^R_M(t)P^R_{N_1}(t)=b_0(M)b_1(N)\frac{1-\gamma({M\otimes_RN})t}{\big(1-\gamma(M)t\big)\big(1-\gamma(N)t\big)}
\]
The desired conclusion about $P_k^R(t)$ is then obtained using the
fact that $\fm(M\otimes_RN_1)=0$, cf.\ Lemma \ref{gammas}(2).
\end{proof}

\section{Rings with $\fm^3=0$}
\label{Rings with m^3=0}

In this section, unless otherwise stated, we assume that $R$ is an
Artinian local ring with $\fm^3=0$. We set $e=\nu(\m)$ and
$a=\dim_k\Soc(R)$.

When $\fm^2=0$, vanishing of homology is not at all mysterious:

\begin{remark}
\label{m^2}
If $\fm^2 =0$ and $\Tor_i^R(M,N)=0$ for some $i>1$, then $M$ or $N$ is
free.

Indeed, assume that $M$ is not free. The module $M_1$ is contained in
$\fm R^{b_0(M)}$, hence $\fm M_1=0$. It is thus a finite sum of copies
of $k$. Since $\Tor_{i-1}^R(M_1,N)=0$, we conclude  that $N$ has finite
projective dimension, hence it is free.
\end{remark}

The behavior of Betti numbers of finitely generated $R$-modules was
studied by Lescot \cite{Le}. The following results are collected from
the proofs of \cite[3.3]{Le}.

\begin{chunk}
 \label{properties}
Assume $M$ is not free and $\fm^2M=0$. For any $i\ge 0$ the following
hold:
\begin{enumerate}[\quad\rm(1)]
\item There is an inequality $b_{i+1}(M)\ge eb_{i}(M)-\nu(\fm M_i)$.
  Equality holds if and only if $k$ is not a direct summand of $M_{i+1}$.
\item If $i>1$ and $k$ is not a direct summand of $M_i$, then $\nu(\fm
  M_i)=ab_{i-1}(M)$.
\end{enumerate}
\end{chunk}

\begin{remark}
\label{socle}
If $\fm^2M=0$ and $k$ is not a direct summand of $M$, then
$\Soc(M)=\fm M$. (This statement holds for all local rings
$R$, not only for those with $\fm^3=0$.)
\end{remark}

\begin{remark} \label{Soc(R)} If $\Tor^R_i(M, N)=0$ for some $i \ge
 3$ and $M$, $N$ are not free, then $\Soc(R)=\m ^2$.  Indeed, it is
 enough to show that $\Soc(R)\subseteq \m ^2$, the other inclusion
 being obvious. Note that $\Soc(M_{i-1})=\Soc(R^{b_{i-2}(M)})$.  On
 the other hand, Remark \ref{summands} shows that $k$ is not a direct
 summand in $M_{i-1}$, hence $\Soc(M_{i-1})=\m M_{i-1}\subseteq \fm^2
 R^{b_{i-2}(M)}$ by Remark \ref{socle}, and the conclusion follows.
 \end{remark}

\begin{theorem}
\label{3tors}
Let $(R,\fm)$ be a local ring with $\fm^3=0$, and let $M$,$N$ be
non-free $R$-modules satisfying $\fm^2M=\fm^2N=0$.

If there exists an integer $j>0$ such that $$
  \Tor^R_j(M, N)=\Tor^R_{j+1} (M, N)=\Tor^R_{j+2}(M, N)=0
$$ then the following hold:
\begin{enumerate}[\quad\rm(1)]
\item $\gamma(M)$ and $\gamma(N)$ are positive integers.
\item $\displaystyle{\frac{b_{i+1}(M)}{b_i(M)}=\gamma(N)}$ and
  $\displaystyle{\frac{b_{i+1}(N)}{b_i(N)}}=\gamma(M)$ for all $i$ with
$0\le i\le j+1$.
\item $\gamma(M)=\gamma(M_i)$ and $\gamma(N)=\gamma(N_i)$ for all
  $i$ with $0\le i\le j$.

\item $\gamma(M) +\gamma(N)=e$ and $\gamma(M)\gamma(N)=a$.
\end{enumerate}
\end{theorem}

\begin{proof}
(1) We will show that $\gamma(M)$, $\gamma(N)$
satisfy the
  equation $\gamma^2-e\gamma+a=0$. As $\gamma(M)$, $\gamma(N)$ are
  positive rational numbers, this implies that they are integers. The
  statement is symmetric in $M$ and $N$, hence it suffices to prove it
  for $\gamma(M)$.

  By Lemma \ref{gammas} we have:
\begin{equation}
\label{eq1}
b_j(N)\le \gamma(M)b_{j-1}(N)\quad\text{and}\quad
b_{i+1}(N)=\gamma(M)b_i(N)\quad\text{for}\quad i=j,j+1
\end{equation}

The hypothesis and Remark \ref{summands} imply that $k$ is not a
direct summand of $M_i$ for any $i$ with $0\le i\le j+1$ and then
\ref{properties} gives the following relations:
\begin{equation}
\label{eq2}
b_{j+1}(N)=eb_{j}(N)-ab_{j-1}(N)\quad\text{and}\quad b_{j+2}(N)\ge
eb_{j+1}(N)-ab_{j}(N)
\end{equation}

Combining the second relations of (\ref{eq1}) and (\ref{eq2}) we obtain
\[
\gamma(M)^2b_j(N)\ge e\gamma(M)b_j(N)-ab_j(N)
\]
Canceling $b_j(N)$ we get $\gamma(M)^2\ge e\gamma(M)-a$.

On the other hand, using the first relation of (\ref{eq2}) and (\ref{eq1})
we have:
\[
\gamma(M)b_j(N)=eb_j(N)-ab_{j-1}(N)\le eb_j(N)-a\frac{b_j(N)}{\gamma(M)}
\]
Canceling $b_j(N)$ and multiplying both sides by $\gamma(M)$ we
obtain $\gamma(M)^2\le e\gamma(M)-a$.

We conclude:
\begin{equation}
\label{eq3}
\gamma(M)^2-e\gamma(M)+a=0\quad\text{and}\quad\gamma(N)^2-e\gamma(N)+a=0
\end{equation}

(2) We show by induction on $j+1-i$ that $b_{i+1}(N)=\gamma(M)b_i(N)$
for all $i$ with $0\le i\le j+1$.  By (\ref{eq1}), the relation holds
for $i=j,j+1$.  Assuming it holds for $i=l+1$, with $0\le l<j$, we
prove that it holds for $i=l$. By \ref{properties} we have
$b_{l+2}(N)=eb_{l+1}(N)-ab_l(N)$, hence, using the inductive
hypothesis and (\ref{eq3}) we obtain
\[
ab_l(N)=\big(e-\gamma(M)\big)b_{l+1}(N)=a\gamma(M)^{-1}b_{l+1}(N)
\]
and the conclusion follows.

(3) Let $i$ be as in the statement and set $l=j+1-i$. The hypothesis
implies
$\Tor_l^R(M_i,N)=\Tor_{l+1}^R(M_i,N)=0$, hence
$b_{l+1}(N)=\gamma(M_i)b_l(N)$ by Lemma \ref{gammas}(3). By (1), we
also have $b_{l+1}(N)=\gamma(M)b_l(N)$, hence $\gamma(M_i)=\gamma(M)$.

(4) By Lemma \ref{gammas}(2) we have $\fm(M\otimes_R N_j)=0$. The
hypothesis implies $\Tor_1(M,N_j)=\Tor_2(M,N_j)=0$, hence, by Lemma
\ref{sum} we get $\gamma(M)+\gamma(N_j)=e+\gamma(M\otimes_RN_j)=e$.
Using (2) we have then $\gamma(M)+\gamma(N)=e$. Recall from (\ref{eq3})
that $\gamma(M)$ and $\gamma(N)$ are roots for the equation
$\gamma^2-e\gamma+a=0$. We obtain:
\begin{align*}
2\gamma(M)\gamma(N)&=\big(\gamma(M)+\gamma(N)\big)^2-\gamma(M)^2-\gamma(N)^2
\\&=e^2-\big(e\gamma(M)-a\big)-\big(e\gamma(N)-a\big)\\
  &=e^2-e\big(\gamma(M)+\gamma(N)\big)+2a=2a
\end{align*}
and this finishes the proof of the theorem.
\end{proof}

\begin{Remark}
  Let $M,N,j$ be as in the statement of the Theorem. If $l\ge j+3$ and
  $k$ is not a direct summand of $M_i$ for all $i<l$
  (In view of Lemma \ref{summands} this happens, for example, when
$\Tor_l^R(M,N)=0$), then
\[
\frac{b_{i+1}(M)}{b_i(M)}=\gamma(N)\quad\text{and}\quad\frac{b_{i+1}(N)}{b_i(N)}=\gamma(M)\quad\text{for
all $i$
  with}\quad 0\le i\le l-1
\]

Indeed, by \ref{properties} we have $b_{i+1}(N)=eb_{i}(N)-ab_{i-1}(N)$
for all $i\le l-2$.  We proceed by induction on $i$, as in the proof
of Theorem \ref{3tors}. By this theorem, the statement is true for all
$i\le j+1$. Assuming that $i\le l-1$ and
$b_{i}(N)=\gamma(N)b_{i-1}(N)$, we then have:
\[
b_{i+1}(N)=eb_{i}(N)-a\frac{b_{i}(N)}{\gamma(M)}=b_{i}(N)\Big(e-\frac{a}{\gamma(M)}\Big)=b_{i}(N)\gamma(M)
\]
where the last inequality is due to the fact that $\gamma(M)$ is a
solution of the equation $\gamma^2-e\gamma+a=0$.
\end{Remark}

Since $R$ is Artinian, it has a dualizing module $\omega$. In the
remaining part of the section we present results that are obtained
when one of the modules $M$, $N$ is equal to $\omega$. An important
part in our arguments is played by Matlis duality. We recall below
some basic facts.

\begin{chunk}
\label{duality&}
Let $R$ be an Artinian ring (not necessarily with $\fm^3=0$) and let
$\omega$ denote its dualizing module.
Matlis duality then gives: $\nu(\omega)=\dim_k\Soc(R)$,
$\dim_k\Soc(\omega)=1$ and $\lambda(\omega)=\lambda(R)$.

\begin{subchunk}
\label{duality}
For every $R$-module $M$ we set $M^\vee=\Hom_R(M,\omega)$. For any
$R$-modules $M$, $N$ and any $i$ there are isomorphisms:
\[
\Tor^R_i(M,N^\vee)\cong \Ext_R^i(M,N)^\vee
\]
\end{subchunk}
\begin{subchunk}
\label{reflexive}
We also set $M^*=\Hom_R(M,R)$. The ring $R$ is then Gorenstein if and
only if $\omega$ is isomorphic to $\omega^{**}$. Indeed, if $R$ is
Gorenstein, then the relation holds trivially.  Conversely, if
$\omega^{**}\cong\omega$, then
\begin{align*}
\omega^*\otimes_R\omega\cong&\Hom_R\big(\Hom_R(\omega^*\otimes_R\omega,\omega),\omega\big)
\cong\Hom_R\big(\Hom_R\big(\omega^*,
  \Hom_R(\omega,\omega)\big),\omega\big)\\
  &\cong\Hom_R(\omega^{**},\omega)\cong \Hom_R(\omega,\omega)\cong R
\end{align*}
It follows that $\omega$ is cyclic, hence $R$ is Gorenstein.
\end{subchunk}
\end{chunk}

We now return to the case of interest, when $\fm^3=0$.

\begin{chunk}
\label{computations}
Assume that $\fm^2\ne 0$. By the above, we have $\nu(\omega)=a$. Since
$\fm^2\omega$ is not zero and is contained in $\Soc(\omega)$, we also
have $\nu(\fm^2\omega)=1$. Setting $N=\omega_1$ and $r=\nu(\fm^2)$, we
can make then the following computations:
\begin{enumerate}
\item $\lambda(\omega)=\lambda(R)=1+r+e$.
\item $\nu(\fm
\omega)=\lambda(\omega)-\nu(\fm^2\omega)-\nu(\omega)=1+r+e-1-a=e+r-a$.
\item $\lambda(N)=(a-1)(1+r+e)$. This follows from a length count in
      the short exact sequence
\[
0\lra N\lra R^{a}\lra \omega\lra 0
\]
\item If $k$ is not a direct summand of $N$ and $a=r$, then
\ref{properties}(1) and (2) give
\begin{gather*}
\nu(N)=e\nu(\omega)-\nu(\fm\omega)=ea-e=e(a-1)\\
\gamma(N)=\frac{\lambda(N)}{\nu(N)}-1=\frac{(a-1)(1+a+e)}{(a-1)e}-1=\frac{1+a}{e}
\end{gather*}
\end{enumerate}
\end{chunk}

\begin{proposition}
\label{3tors-o}
Let $(R,\fm)$ be an non-Gorenstein Artinian ring with $\fm^3=0$ and $M$ a
non-free finitely generated $R$-module with $\fm^2M=0$.

If there exists an integer $j\ge 2$ such that
\[
\Tor_j^R(M,\omega)=\Tor_{j+1}^R(M,\omega)=\Tor_{j+2}^R(M,\omega)=0
\]
then $e=a+1$, $\gamma({\omega_1})=1$, $\gamma(M)=a$ and
$b_0(M)=b_1(M)=\dots =b_{j+2}(M)$.
\end{proposition}

\begin{proof}
  Set $N=\omega_1$. By Remark \ref{Soc(R)} we have $\Soc(R)=\fm^2$.
  Also, $k$ is not a direct summand of $N$ by Remark \ref{summands},
  hence \ref{computations}(4) gives $\gamma(N)=(1+a)/e$.

By Theorem \ref{3tors}(4), $\gamma(N)$ is a solution of the equation
$\gamma^2-e\gamma+a=0$. It follows that $(a+1)^2=e^2$, hence $a+1=e$.
In particular, $\gamma(N)=1$, and Theorem \ref{3tors}(4) implies
$\gamma(M)=a$.  The conclusion about the Betti numbers follows from
Theorem \ref{3tors}(2).
\end{proof}

Note that there are examples when the situation in Proposition
\ref{3tors-o} holds, $M$ is not free, and $j$ can be chosen to be
arbitrarily large. Such an example is provided by Avramov, Gasharov
and Peeva \cite[2.2]{AGP}, as described below.

 If $F$ is a complex, then $F^*$ denotes the complex $\Hom_R(F,R)$, with
induced differentials.

\begin{example}
\label{Oana}
Let $l$ be a field, let $X=\{X_1,X_2,X_3,X_4\}$ be a set of
indeterminates over $l$ and set $A=l[X]_{(X)}$. Let $I$ be the ideal
of $A$ generated by elements
\begin{gather*}
X_1^2,X_1X_2-X_3X_4, X_1X_2-X_4^2, X_1X_3-X_2X_4, X_1X_4-X_2^2,\\
X_1X_4-X_2X_3, X_1X_4-X_3^2
\end{gather*}
and set $R=A/I$. The ring $R$ is then local and has $\fm^3=0$.

Let $x_i$ denote the image of $X_i$ in $R$ for $i=1,\dots,4$ and
consider the sequence of homomorphisms of free $R$-modules:
$$F=\qquad \dots \xrightarrow{\psi}
R^2\xrightarrow{\varphi}R^2\xrightarrow{\psi}R^2\xrightarrow{\varphi}\dots$$
where
\begin{align*}
\varphi=\left(\begin{matrix}
x_3& x_1\\
x_4&x_2\end{matrix}\right)
\qquad\psi=\left(\begin{matrix}
x_2& -x_1\\
-x_4&x_3\end{matrix}\right)
\end{align*}

Set $M=\Coker \varphi$. By \cite[(2.2)(i)]{AGP} the complex $F$ is
exact. As noted by Veliche \cite{Ve}, a computation similar to one in
\cite[Section 3]{AGP} shows that the complex $F^*$ is exact. This
yields $\Ext^i_R(M,R)=0$ for all $i>0$, or equivalently, cf.\
\ref{duality}, $\Tor_i^R(M,\omega)=0$ for all $i>0$.
\end{example}

\begin{theorem}
\label{omega}
Let $(R, \fm )$ be an Artinian local ring with $\fm ^3=0$.  The
following statements are equivalent:
\begin{enumerate}[\quad\rm(1)]
\item $R$ is Gorenstein.
\item $\Tor_1^R(\omega, \omega )=0$.
\item $\Tor_2^R(\omega, \omega)=\Tor_3^R(\omega, \omega)=0$.
\item $\Tor_j^R(\omega, \omega)=\Tor_{j+1}^R(\omega,
      \omega)=\Tor_{j+2}^R(\omega, \omega)=0$ for some $j \ge 3$.
\end{enumerate} \end{theorem}

We recall a result of Asashiba and Hoshino \cite[2.1]{AsHo}:

\begin{chunk}
\label{maps}
Let $R$ be a local ring. If $M$ is a faithful $R$-module and the sequence
\[
0\lra N\xrightarrow{\ \varphi\ } R^2\xrightarrow {\ \psi\ } M\lra 0
\]
is exact, then there exist homomorphisms $\alpha$ and $\beta$ making
the following diagram commute:
\[
\xymatrixrowsep{1.8pc} \xymatrixcolsep{2.2pc}
\xymatrix{
0\ar@{->}[r]&
N\ar@{->}[r]^{\varphi}\ar@{->}[d]^{\alpha}&R^2\ar@{->}[d]^{\theta}\ar@{->}[r]^{\psi}&M\ar@{->}[d]_{\beta}\ar@{->}[d]\ar@{->}[r]&0\\
0\ar@{->}[r]&
M^*\ar@{->}[r]^{\psi^*}&(R^2)^*\ar@{->}[r]^{\varphi^*}&N^*&{}
}
\]
\end{chunk}

\setcounter{theorem}{10}
\setcounter{equation}{0}

\begin{proof}[Proof of Theorem {\rm\ref{omega}}]
Set $r=\nu(\fm^2)$ and $N=\omega_1$ and $b_i=b_i(\omega)$.

The implications (1) $\Rightarrow $ (2), (1) $\Rightarrow $
(3), (1) $\Rightarrow $ (4) are obvious.

 (2) $\Rightarrow$ (1) Assume $R$ is not Gorenstein.
  By \ref{computations} we have $\lambda(\omega)=1+e+r$ and
$\nu(\fm\omega)=e+r-a$. We obtain then
\[
\gamma(\omega)=\frac{1+e+r}{a}-1=\frac{1+e+r-a}{a}
\]
By \ref{properties}(1) we have $b_1\ge
eb_0-\nu(\fm\omega)=ea-(e+r-a)$ and by Lemma \ref{gammas} we have
$b_1\le \gamma(\omega) b_0$. We conclude:
\[
ea-(e+r-a)\le\frac{1+e+r-a}{a}\cdot a=1+e+r-a
\]
hence $ea-2e+2a-2r\le 1$, or, equivalently, $e(a-2)+2(a-r)\le 1$.
Since $e>1$ and $a\ge r$ we conclude that $a=2=r$ and $\nu(N)=b_1\ge
e$. By \ref{computations} we have thus
\[
\lambda(R)=\lambda(N)=e+3
\]

The hypothesis implies there is a short exact sequence
\[
0\to N\otimes_R\omega\to \omega^2\to \omega\otimes_R\omega\to 0\,.
\]
with $\lambda(\omega^2)=2(e+3)$,
$\lambda(N\otimes_R\omega)=2\nu(N)+\varepsilon$ and
$\lambda(\omega\otimes_R\omega)=4+\eta$, where $\varepsilon$ and
$\eta$ are nonnegative integers. A length count in the short exact
sequence then gives:
\begin{equation}
\label{epsilon}
2e+6=2\nu(N)+\varepsilon+4+\eta\ge 2e+4+\varepsilon+\eta
\end{equation}

In particular, it follows that $\eta\le 2$, hence
$\lambda(\omega\otimes_R\omega)=4+\eta\le 6$. Note that
$\lambda(\omega^*)=\lambda(\omega\otimes_R\omega)$, as we have
$\Hom_R(\omega\otimes_R\omega,\omega)\cong \omega^*$.

For the rest of the proof we will look at the commutative diagram in
\ref{maps}, with $M=\omega$. Note that $\alpha\col N\to \omega^*$ is
injective. Also, the lower sequence in the diagram is right-exact, by
the hypothesis $\Ext^1_R(\omega,R)=0$. In particular,
$\beta\col\omega\to N^*$ is surjective, or equivalently, the dual map
$\beta^\vee\col N\otimes_R\omega\to R$ is injective.

Since $N$ is contained in $\omega^*$, we have $e+3=\lambda(N)\le
\lambda(\omega^*)=4+\eta$ and thus $e\le \eta +1\le 3$.  In particular, we
have
$\eta\in\{1,2\}$.

If $\eta=1$, then $e=2$, hence $\lambda(N)=\lambda(\omega^*)=5$. It
follows that $\alpha$ is an isomorphism. The commutative diagram in
\ref{maps} yields that $\beta$ is an isomorphism, hence $\omega\cong
N^*\cong\omega^{**}$. We apply then \ref{reflexive} to conclude
that $R$ is Gorenstein, a contradiction.

If $\eta=2$, then the inequality (\ref{epsilon}) yields
$\varepsilon=0$ and $\nu(N)=3$, hence $e\le 3$ and
$\lambda(N\otimes_R\omega)=6$.  On the other hand, as noted above,
$N\otimes_R\omega$ is contained in $R$. As $\lambda(R)=e+3\le 6$, it
follows that $N\otimes_R\omega\cong R$, hence $\omega$ is cyclic,
contradicting our assumption that $R$ is not Gorenstein.

 (3) $\Rightarrow$ (1) Assume $R$ is not Gorenstein. By Remark
\ref{Soc(R)}, we have
  $\Soc(R)=\fm^2$ and then \ref{computations} gives
  $\lambda(\omega)=\lambda(R)=1+a+e$ and $\nu(\fm\omega)=e$.

  Since $\Tor_1^R(N,\omega)=\Tor_2^R(N,\omega)=0$, we use Theorem
  \ref{gammas}(3) to obtain:
\[
b_2=\gamma(N) b_1=\frac{\lambda(\fm N)}{\nu(N)}b_1=\frac{\nu(\fm
N)}{b_1}b_1=\nu(\fm N)
\]
On the other hand, \ref{properties}(1) gives $b_2=eb_1-\nu(\fm
N)=eb_1-b_2$, hence we have:
\[
2b_2=eb_1\quad\text{and}\quad \nu(\fm N)=\frac{eb_1}{2}
\]
and we conclude $\gamma(N)=e/2$. By Lemma \ref{gammas}(1) we have then
\[
b_1\le \gamma(N)b_0=\frac{e}{2}a
\]
By \ref{properties}(1) we also have $b_1\ge
eb_0-\nu(\fm\omega)=e(a-1)$. We obtain thus
\[
e(a-1)\le \frac{e}{2}a
\]
and we conclude $a\le 2$. As we assumed $R$ not to be Gorenstein, we
have $a=2$. Recall from \ref{computations} that
$\gamma(N)=(a+1)/e$. Comparing this with the relation
$\gamma(N)=e/2$ obtained above, we obtain $e^2=6$, a
contradiction.

(4) $\Rightarrow $ (1)
  Assume that $R$ is not Gorenstein and set $N=\omega_1$. Applying
  Proposition \ref{3tors-o} with $M=N$ we obtain $\gamma(N)=1$. We
  then use Proposition \ref{3tors}(4) with $M=N$ and we conclude
  $a=\gamma(N)^2=1$, hence $R$ is Gorenstein, a contradiction.
\end{proof}

When $(R,\fm)$ is a local commutative Noetherian ring, the following
question has been considered in the literature and is still open: Does
there exist a number $d$, depending only on $R$, such that whenever
$M$, $N$ are finiteley generated $R$-modules with $\Tor_i^R(M,N)=0$
for all $i\gg 0$ it follows that $\Tor_i^R(M,N)=0$ for all $i>d$?

This problem was considered for Gorenstein rings by Huneke and
Jorgensen \cite{HJ}. A positive answer is known for complete
intersection rings, Gorenstein rings with $\fm^3=0$, and Gorenstein
rings of codimension at most $4$, cf.\ \cite{AB}, \cite{HJ},
respectively \cite{Se}. Our results seem to point towards a positive
answer for all Artinian rings with $\fm^3=0$, but fall short of a
proof. The following remark provides some insight.

\begin{remark} Let $M$, $N$ and $j$ be as in the statement of
Theorem \ref{3tors} and let $i$ be an integer with $0\le i\le j$.  Then
$\Tor_{i+1}(M,
N)=0$ if and only if $\fm(M_i \otimes N)=\fm(M_{i+1}\otimes N)=0$.

Indeed, consider the exact sequence
$$
0\lra \Tor_{i+1}(M, N)\lra M_{i+1}\otimes_R N \lra N^{b_i(M)} \lra M_i
\otimes N \lra 0
$$
Theorem \ref{3tors} shows that
\[
\lambda (N) b_i(M)= \nu
(N)b_{i+1}(M) + \nu(N)b_i(M) = \nu (M_{i+1} \otimes_R N)+\nu (M_i
\otimes _R N)
\]
On the other hand $\fm(M_i \otimes _R N)=0$ and $\fm(M_{i+1}\otimes _R
N)=0$ if and only if $\nu (M_i \otimes _R N)= \lambda (M_i\otimes _R
N)$, and $\nu (M_{i+1} \otimes_R N)=\lambda (M_{i+1} \otimes_R N)$,
respectively. Counting lengths in the above exact sequence gives the
desired conclusion.
\end{remark}

\section{Rings of large embedding dimension} \label{Rings with large
embedding dimension}

For any finitely generated module $N$ we set
\[
c(N)=\max\{4, \log_2(b_1(N))+2\}
\]
(where $\log _20 =-\infty$).

The {\it Loewy length} of the ring $R$, denoted $\ell\ell(R)$, is the
largest
integer $h$ for which $\fm^h\ne 0$.

In this section we prove the following:

\begin{theorem}
\label{adela}
Let $(R,\fm)$ be an Artinian local ring satisfying $\nu(\fm)\ge
\lambda(\fm^2)-\ell\ell(R)+4$.
\begin{enumerate}[\quad\rm(1)]
\item
If $\fm^2M=0$ and $\Tor_i^R(M,N)=0$ for all $i\in [1, c(N)]$, then
either $M$ or $N$ is free.\item If $\m^3=0$ and $\Tor_i^R(M,N)=0$ for
three consecutive values of $i\ge 2$, then either $M$ or $N$ is free.
\end{enumerate}
\end{theorem}

\begin{proof}[Proof of Theorem {\rm\ref{adela}}]
 Set $e=\nu(\fm)$ and $h=\ell\ell(R)$. Assume that $M$ and $N$ are not
free.

(1) Let $i$ be a positive integer. In the proof \cite[2.2]{GP}
    Gasharov and
Peeva show that the following inequality holds for a finitely generated
module
$N$ and any local Artinian ring $R$:
\[
b_{i+1}(N)\ge eb_i(N)-\big(\lambda(\fm^2)+2-h\big)b_{i-1}(N)
\]
Setting $a=\lambda(\fm^2)+3-h$, we conclude that for all positive
integers $i$ there is a strict inequality
\[
b_{i+1}(N)>eb_i(N)-ab_{i-1}(N)
\]

We let then $i$ be any integer such that $2\le i\le c(N)-2$. Lemma
\ref{gammas}(3) gives that $b_{j}(N)=\gamma(M)^{j-1}b_{1}(N)$, for
$j=i,i+1$ and we conclude
\begin{equation}
\label{strict}
\gamma(M)^2-e\gamma(M)+a>0
\end{equation}

The roots of the equation $\gamma ^2-e\gamma +a=0$ are
\[
\gamma_{1,2}=\frac{e \pm\sqrt{e^2-4a}}{2}
\]
The hypothesis gives $a\le e-1$, hence $e^2-4a\ge (e-2)^2$. Both
$\gamma_1$ and $\gamma_2$ are then real. Assume $\gamma_1\le \gamma_2$. We
obtain then $\gamma_1\le 1$ and $\gamma_2\ge e-1$.

The strict inequality in (\ref{strict}) shows that $\gamma(M)$ is
outside the interval $[\gamma_1,\gamma_2]$.  If $\gamma(M)<\gamma_1$,
then $\gamma(M)<1$ and this contradicts Proposition \ref{one}. We
conclude that $\gamma(M)>\gamma_2$, hence $\gamma(M)>e-1$. We recall
that $\gamma(M)$ is an integer, cf. Proposition \ref{one}, and we
conclude $\gamma (M)\ge e$.  On the other hand, Lemma
\ref{relations}(1) implies $e>\gamma(M)$, a contradiction.

(2) We replace $M$ with $M_1$, if necessary, so that we may assume
    $\fm^2M=0$. Proceed then as in (1), using Theorem \ref{3tors}.
\end{proof}

\section{The Auslander-Reiten Conjecture}
\label{The Auslander-Reiten conjecture}

In this section we prove the conjecture of Auslander and Reiten
(stated in the introduction) when the module is annihilated by
$\fm^2$. More precise statements are obtained when $\fm^3=0$. We state
our main results below. The proofs will follow later in the section.

\begin{theorem}
\label{AR} Let $(R,\fm)$ be an Artinian local ring with $\fm^3=0$ and
$M$ a finitely generated $R$-module.
\begin{enumerate}[\quad\rm(1)]
\item If $\Ext^i_R(M, M\oplus R)=0$ for four consecutive
values of $i$ with $i\ge 2$, then $M$ is free.
\item If $R$ is Gorenstein and $\Ext^i_R(M,M)=0$ for some $i>0$, then
      $M$ is free
\end{enumerate}
\end{theorem}

The second part was inspired by the following statement of Hoshino
\cite{Hoshino}: If $R$ is a finite dimensional local algebra (possibly
non-commutative) over a field and the cube of its radical is zero,
then any finitely generated $R$-module $M$ with $\Ext^1_R(M,M)=0$ is
free.

\begin{theorem}
\label{m^2M}

Let $(R,\fm)$ be an Artinian local ring and $M$ a finitely generated
$R$-module with $\fm^2M=0$.

If $\Ext^i_R(M, M\oplus R)=0$ for all $i$ with $0<i\le \max \lbrace 3,
\nu (M), \nu (\fm M)\rbrace $, then $M$ is free.  \end{theorem}

\begin{remark}
\label{inverse}
Assume that $\fm^2\ne 0$, $M\ne 0$ and $\fm^2M=0$. If
$\Ext_R^i(M,M)=0$ for some $i>0$, then
$\gamma(M^\vee)=\gamma(M)^{-1}$.

Indeed, \ref{duality} gives $\Tor_i^R(M,M^\vee)=0$. Since $\fm^2$ is
not zero, the modules $M^\vee$ and $M$ are not free, hence $k$ is not
a direct summand in either of them, cf.\ Remark \ref{summands}. Set
$r(M)=\rank_k\Soc(M)$. As noted in Remark \ref{socle},
$r(M)=\lambda(\fm M)$ and $r(M^\vee)=\lambda(\fm M^\vee)$. Using
Matlis duality we obtain
\[
\gamma(M^\vee)=\frac{\lambda(\fm
M^\vee)}{\nu(M^\vee)}=\frac{r(M^\vee)}{r(M)}=\frac{\nu(M)}{\lambda(\fm
M)}=\frac{1}{\gamma(M)}
\]
\end{remark}

\begin{proposition}
\label{Ext}
Let $(R,\fm)$ be an Artinian local ring. Let $M$ be a non-zero
finitely generated $R$-module such that $\fm^2M=0$. If any of the
following conditions holds:
\begin{enumerate}[\quad\rm(1)]
 \item $\Ext^i_R(M,M)=0$ for all $i$ with $0<i\le\max\{3,\nu(M),\nu(\fm
M)\}$
\item $\fm^3=0$ and $\Ext^{i}_R(M,M)=0$ for three consecutive
      values of $i>0$,
\end{enumerate}
then $\fm^2=0$, and $M$ is either free or injective.
\end{proposition}

\begin{proof} By \ref{duality} we have $\Tor_i(M,M^\vee)=0$ for all
$i$ as in the statement.

If $\fm^2=0$, then Remark \ref{m^2} implies $M$ or $M^\vee$ is free,
hence $M$ is free or injective.

>From now in we assume $\fm^2\ne 0$. This implies, in particular, that
neither $M$ nor $M^\vee$ is free. By Remark \ref{summands}, $k$ is not
a direct summand in $M$ or $M^\vee$.  Matlis duality and Remark
\ref{socle} then yield $\nu(\fm M)=\nu(M^\vee)$, while Remark
\ref{inverse} gives $\gamma({M^\vee})=\gamma(M)^{-1}$.  Proposition
\ref{one}, respectively Theorem \ref{gammas}(1), give then
$\gamma(M)\ge 1$ and $\gamma({M^\vee})\ge 1$, and we conclude
$\gamma(M)=\gamma({M^\vee})=1$.

We use the notation of the previous sections: $e=\nu(\fm)$ and
$a=\dim_k\Soc(R)$.

Assume that $M$ satisfies (1). By Lemma \ref{sum} we have
$2-\gamma(M\otimes_RM^\vee)=e$, hence $e\le 2$.

By Scheja \cite[Satz 9]{Sch} the ring $R$ is then either a complete
intersection, or a Golod ring. If it is a complete intersection, then
$\Ext^2_R(M,M)=0$ implies $M$ is free, by Auslander, Ding, and Solberg
\cite[(1.8)]{ADS}, a contradiction. If it is Golod, but not a
hypersurface, then $e=2$ and $\gamma(M\otimes_RM^\vee)=0$.  Lemma
\ref{gammas}(1) yields then $b_i(M)=\nu(M)$ and
$b_i(M^\vee)=\nu(M^\vee)$ for all $i=1,2,3$. Since $R$ is Golod,
$P^R_k(t)=(1+t)(1-t-lt^2)^{-1}$ with $l\ge 1$, hence $b_3(k)=2+3l\ge
5$. Since $M\otimes_RM^\vee$ is a sum of copies of $k$, we have then
$$\beta_3(M\otimes_RM^\vee)=b_3(k)\nu(M)\nu(M^\vee)\ge
5\nu(M)\nu(M^\vee)$$ On the other hand, Lemma \ref{product} gives
$\beta_3(M\otimes_RM^\vee)=4\nu(M)\nu(M^\vee)$, and this leads to a
contradiction.

Assume that $M$ satisfies (2). Using \ref{3tors}(4) we obtain $a=1$
and $e=\gamma(M)+\gamma({M^\vee})=2$. Thus, $R$ is Gorenstein and it
follows that it is a complete intersection. Let $j$ be an even
integer among the three consecutive integers in the hypothesis. The
hypothesis that $\Ext^{j}_R(M,M)=0$ implies $M$ is free, cf. Avramov
and Buchweitz \cite[4.2]{AB}, a contradiction.
\end{proof}

\begin{proof}[Proof of Theorem {\rm \ref{m^2M}}] Assume that $M$ is
not free. Proposition \ref{Ext}(1) then shows $\fm ^2=0$ and $M$ is
injective. By \ref{duality}, $\Ext^i_R(M, R)=0$ implies $\Tor^R _i(M,
\omega)=0$ and Remark \ref{m^2} shows that $\omega$ is free, hence $R$
is Gorenstein. In this case, any finiteley generated $R$-module, and
in particular $M$, is also free, a contradiction.
\end{proof}

The proof of the first part of Theorem \ref{AR} is similar, so we
give it here:

\begin{proof}[Proof of Theorem {\rm \ref{AR}(1)}] The hypothesis
implies $\Ext^i_R(M_1,M_1)=0$ for three
    consecutive values of $i>0$. Since $\fm^2M_1=0$, we proceed as
    above, using Proposition \ref{Ext}(2).
\end{proof}

In order to prove part (2) of Theorem \ref{AR} we use the fact that
for Gorenstein Artinian rings one can define negative Betti numbers
and negative syzygies.

\begin{chunk}
\label{negative}
Let $R$ be a Gorenstein Artinian local ring, $F$ a free resolution of
$M$ and $G$ a free resolution of $M^*$. Note that the complex $G^*$ is
acyclic, with $\HH_0(G^*)=M^{**}\cong M$. Gluing together the
complexes $F$ and $G^*$, we obtain an exact complex $P$, which is
called a {\it
  complete} resolution of $M$.

Furthermore, if $M$ is a first syzygy in a minimal free resolution of
some other module, and the resolutions $F$ and $G$ are chosen to be
minimal, then the complex $P$ is minimal.  In this case, we have
$\rank(P_i)=b_i(M)$ for all $i\ge 0$ and $M_i=\Coker\dd_{i}^P$. In
general, the Betti numbers and syzygies of $M$ are defined by setting
$b_i(M)=\rank(P_i)$ and $M_i=\Coker\dd_{i}^P$ for all integers
$i$. Note that for any $j\ge i$ the module $M_j$ is a $(j-i)$'th
syzygy of $M_i$; in our notation: $M_j=(M_i)_{j-i}$.
\end{chunk}

\setcounter{theorem}{1}
\setcounter{subchunk}{0}

\begin{proof}[Proof of Theorem {\rm \ref{AR}(2)}]
  Assume that $M$ is not free. Since $R$ is Gorenstein, the hypothesis
  implies $\Ext^i_R(M_1,M_1)=0$. Replacing $M$ by $M_1$, we may
  assume $\fm^2M=0$. We can now use the notation of \ref{negative}.
  The assumption that $M$ is not free implies that both $M$ and $M^*$
  have infinite projective dimension, hence $M_j$ is not free for any
  $j$.

  For all $j$ we have $\Ext^i_R(M_j,M_j)=0$ and $\fm^2 M_j=0$. By
  \ref{duality} we then have $\Tor_i^R(M_j,M_j^*)=0$. (Since $R$ is
  Gorenstein, we have $\omega\cong R$, hence $M_j^*\cong
  M_j^\vee$). In particular, $k$ is not a direct summand in any of
  the $M_j$'s.

  We set $e=\nu(\fm)$ and $b_j=b_j(M)$ for all $j$. Since $R$ is
  Gorenstein, $\Soc(R)$ is $1$-dimensional. We use Lescot's results
  recalled in \ref{properties} to get:
\begin{equation}
\label{recurrence}
b_j+b_{j-1}=eb_{j-2} \quad \text{and}\quad \nu(\fm M_j)=b_{j-1}\quad
\text{for all}\quad j
\end{equation}

Recall that $\gamma(M)=\lambda(\fm M)/\nu(M)$. Using
(\ref{recurrence}) and Remark \ref{inverse} we have:
\begin{equation}
\label{gamma}
\gamma(M_j)=\frac{b_{j-1}}{b_j} \quad \text{and}\quad
\gamma(M_j^*)=\frac{b_{j}}{b_{j-1}} \quad \text{for all}\quad j
\end{equation}

Since $\Tor_i^R(M_j,M_j^*)=0$ for all $j$, Lemma \ref{gammas} yields:
\begin{equation}
\label{betta}
b_{i+j}=\big(\gamma(M_j^*)-\gamma(M_{j+i-1}\otimes_R M_j^*)\big)b_{j+i-1}
\end{equation}

For all $j$ we obtain:
\[
b_{j+i}\le \gamma(M_j^*)b_{j+i-1}=\frac{b_{j}}{b_{j-1}}\, b_{j+i-1}
\]
where the inequality comes from (\ref{betta}) and the equality from
(\ref{gamma}). Equivalently:
\[
\frac{b_{j+i}}{b_{j+i-1}}\le \frac{b_j}{b_{j-1}} \quad\text{for
all}\quad j
\]

Each $j=0, 1,\cdots, i-1$ yields thus a non-increasing sequence $(
b_{j+ni}/b_{j+ni-1})_{n}$. Let $L_j$ denote the limit of the $j$'th
sequence.

If $L_j <1$ for some $j=0, 1, \ldots, i-1$ it follows that there
exists an eventually strictly decreasing subsequence of $\{ b_n \}$.
This implies that $b_n=0$ for some $n\gg 0$, a contradiction.

Thus, $L_j \ge 1$ for all $j$. This implies $b_n\le b_{n+1}$ for all
$n$. If $b_{n_0} < b_{n_0+1}$ for some $n_0$, then we obtain
$$
1 < \frac{b_{n_0 +1}}{b_{n_0}} \le \frac{b_{n_0 -i +1}}{b_{n_0-i}}
\le \cdots
$$
hence
\[
b_{n_0+1} > b_{n_0} \ge b_{n_0-i+1} > b_{n_0 -i } \ge
b_{n_0-2i+1}> b_{n_0-2i} \ge \cdots
\] It follows that $b_n=0$ for some $n\ll 0$, a contradiction. In
conclusion, $b_n=b_{n+1}$ for all $n$.  We use then (\ref{recurrence})
to obtain $e=2$. It follows that $R$ is a complete intersection. Since
the Betti numbers of any $M_j$ are constant, by a result of Eisenbud
\cite{Ei} then $M_{j+n}\cong M_{j+n+2}$ for all $n>0$. We conclude
$M_j\cong M_{j+2}$ for all $j$.

If $i$ is even, then $M$ has finite projective dimension by \cite{AB},
hence it is free, a contradiction.

If $i$ is odd, then $M\cong M_{-i+1}$, hence
\[
\Ext^1_R(M,M)\cong\Ext^1_R(M,M_{-i+1})\cong\Ext^i_R(M,M)=0
\]

Since the Betti numbers of $M$ are constant, (\ref{gamma}) gives
$\gamma(M^*)=1$. Taking $i=1$ and $j=0$ in (\ref{betta}) we get
$b_{1}=\big(1-\gamma(M\otimes_R M^*)\big)b_{0}$, and it follows
$\gamma(M\otimes_R M^*)=0$. This means that $\fm(M\otimes_RM^*)=0$.
However, $M\otimes_R M^*$ is the Matlis dual of $\Hom_R(M,M)$ and the
later is annihilated by $\fm$ only when $\fm M=0$. In view of the
hypothesis, this implies that $M$ is free,which provides the desired
contradiction.
\end{proof}

\section{Vanishing of Tor and Loewy length} \label{Vanishing of Tor
and the Loewy length}

In this section we assume that $(R,\fm,k)$ is an Artinian local ring.
We recall that the Loewy length of $R$, denoted $\ell\ell(R)$, is the
largest integer $h$ with $\fm^h\ne 0$.

We propose the following conjecture:

\begin{conjecture}
\label{conjecture}
Assume that $M$, $N$ are nonzero modules with $\fm^2M=\fm^2N=0$.

If $\Tor_i^R(M,N)=0$ for all $i>0$, then $\fm^3=0$.
\end{conjecture}

\begin{chunk}
  One can ask a more general question: Assume that $M$, $N$ are
  nonzero modules with $\fm^pM=\fm^qN=0$. If $\Tor_i^R(M,N)=0$ for all
  $i>0$, does it follow that $\fm^{p+q-1}=0$? This is trivially true
  when $p=1$ or $q=1$, or when one of $p$, $q$ is greater than
  $\ell\ell(R)$. The conjecture takes up the case $p=2=q$.
\end{chunk}

\begin{chunk}
  In view of Lemma \ref{Poincare}, the conjecture holds whenever
  $\fm^3=0$ or $P^R_k(t)\ne (1-at)(1-bt)^{-1}(1-ct)^{-1}$ with
  rational numbers $a\ge 0$ and $b,c>0$.  The class of such rings
  include: complete intersection rings of codimension different from
  $2$, generalized Golod rings, Koszul rings, rings with irrational
  Poincar\'e series and many others. More evidence for the conjecture
  can also be gathered from the preceding two sections.
\end{chunk}

Theorem \ref{graded} below establishes the conjecture in yet another
important case.

We say that the local ring $R$ is standard graded if it has a
decomposition $R=R_0\oplus R_1\oplus\dots\oplus R_h$ such that
$R_iR_j\subseteq R_{i+j}$ for all $i,j\in[0,h]$, $R_0=k$ and
$R=R_0[R_1]$ (it is thus generated in degree one).

\begin{theorem}
\label{graded}
Let $R$ be a standard graded local ring, and let $M$, $N$ be non-zero
finitely generated $R$-modules satisfying $\fm^2M=\fm^2N=0$.

If $\Tor_i^R(M,N)=0$ for all $i>0$, then
$\fm^3=0$.
\end{theorem}

Let $L$, $U$ be any two $R$-modules. We set $\ov L=L/\fm L$ and we
consider the
short exact sequence
$$
0\lra \fm L\xrightarrow{\mu_{L}} L\lra \ov L\lra 0
$$
and the induced long exact sequence:

$$\dots\to\Tor_{i+1}(U,\ov L)\xrightarrow{\Delta_i(U,L)}\Tor_i^R(U,\fm
L)\xrightarrow{\Tor_i^R(U,\mu_L)}\Tor_i^R(U,L)\to\dots $$ where
$\Delta_i(U,L)$ denote the connecting homomorphisms.

\begin{lemma}
\label{zero}
Assume that $\fm^2M=0$ and $N$ is not free. Let $j\ge 2$ be an
integer.

If $\Tor_i^R(M,N)=0$ for all $i\in [1,j]$, then $\Tor_i^R(k,\mu_N)=0$
for all $i\in [0,j-1]$.

\end{lemma}

\begin{proof}

  (1) We will show, equivalently, that $\Delta_i(k,N)$ is surjective
      for all $i\in[0,j-1]$.
  Of course, this is true when $i=0$. We prove the claim by induction.
  Assume that it is true for $i=n$ and we prove it for $i=n+1$, with
  $0\le n<j-1$.

  In the diagram below the horizontal lines are long exact sequences
  of the type considered above.
\[
\xymatrixrowsep{2.2pc} \xymatrixcolsep{0.9pc}
\xymatrix{
\Tor^R_{n+2}(M,\ov{N})\ar@{->}[r]\ar@{->}[d]^{\cong}_{\Delta_{n+1}(M,N)}&\Tor^R_{n+2}(\ov
M,
\ov{N})\ar@{->}[r]\ar@{->}[d]_{\Delta_{n+1}(\ov M,N)}&\Tor^R_{n+1}(\fm
M,\ov{N})\ar@{->}[r]\ar@{->>}[d]_{\Delta_n(\fm M,N)}&\Tor^R_{n+1}(M,
\ov{N})\ar@{^{(}->}[d]_{\Delta_n(M,N)}\\
\Tor^R_{n+1}(M,\fm N)\ar@{->}[r]&\Tor^R_{n+1}(\ov M, \fm
N)\ar@{->}[r]&\Tor^R_n(\fm M, \fm N)\ar@{->}[r]&\Tor_n^R(M,\fm N)
}
\]
By \cite[?]{CE} the exterior squares commutes and the interior one
anticommutes. The map $\Delta_{n+1}(M,N)$ is bijective because
$\Tor_{n+2}^R(M,N)=\Tor_{n+1}^R(M,N)=0$. Also, the map $\Delta_n(\fm
M,N)$ is surjective by the induction hypothesis, using the fact that
$\fm M$ is a finite direct sum of copies of $k$, and $\Delta_n(M,N)$
is injective because $\Tor_{n+1}^R (M,N)=0$. By the ``Five Lemma'' we
conclude that the map $\Delta_{n+1}(\ov M,N)$ is surjective.  Since
$\ov M$ is a finite direct sum of copies of $k$, we obtain that
$\Delta_{n+1}(k,N)$ is surjective, and this finishes the induction
argument.
\end{proof}

\begin{proof}[Proof of Theorem {\rm \ref{graded}}]

  By Lemma \ref{zero} we have $\Tor_i^R(k,\mu_{N_1})=0$ for all $i\ge
  0$. Choose such an $i$ and set $F=R^{b_0(N)}$. By \ref{relations}(2)
  we have $\fm N_1=\fm^2 F$. Consider the following commutative
  diagram, in which the vertical maps are induced by the inclusion
  $N_1\into \fm F$:
\[
\xymatrixrowsep{2.0pc} \xymatrixcolsep{5.7pc}
\xymatrix{
\Tor_i^R(k, \fm N_1)
\ar@{->}[r]^{\Tor_i^R(k,\mu_{N_1})=0}\ar@{->}[d]^{\cong}&\Tor_i^R(k,
N_1)\ar@{->}[d]\\
\Tor_i^R(k,\fm^2F)\ar@{->}[r]^{\Tor_i^R(k,\mu_{\fm F})}&\Tor_i^R(k,\fm F)
}
\]
We conclude that $\Tor_i^R(k,\mu_{\fm F})=0$, hence
$\Tor_i^R(k,\mu_{\fm})=0$ for all $i$. Equivalently, the map
$\Ext^i_R(k,k)\to\Ext^i_R(R/\fm^2,k)$ induced by the projection $R\to
R/\fm^2$ is zero for all $i$, hence \cite [Corollary 1] {Ro} implies
that the algebra $\Ext_R^*(k,k)$ (with Yoneda product) is generated by
its elements of degree $1$.  This means that the $k$-algebra $R$ is
Koszul, hence $P_k^R(t)=\Hilb_R(-t)^{-1}$, cf.\ \cite[Theorem
1.2]{Lo}. Comparing with the relation of Remark \ref{Poincare} we
conclude that $\Hilb_R(t)$ is a polynomial of degree at most $2$,
hence $\fm^3=0$.
\end{proof}

\section{Rings of type at most $2$}
\label{Rings of type at most 2}

In this section we show that the conjecture of Tachikawa stated in the
introduction holds for Cohen-Macaulay rings of type at most $2$.

\begin{theorem}
\label{tor22}
Let $R$ be a Cohen-Macaulay local ring with a canonical module
$\omega$ and such that $\type(R)\le 2$.
\begin{enumerate}[\quad\rm(1)]
\item If $\Tor_2^R(\omega,\omega)=0$, then $R$ is Gorenstein.
\item If $\Ext^i_R(\omega, R)=0$ for $i=1,2$, then $R$ is Gorenstein.
\end{enumerate} \end{theorem}

The proof will be given at the end of the section, after discussing
some preliminaries. We recall a well-known fact:

\begin{chunk}
\label{wedge}
  Let $R$ be a commutative local ring and consider a short exact
  sequence of $R$-modules:
\[
0\lra N\xrightarrow{\ \varphi\ } R^n\xrightarrow {\ \psi\ } M\lra 0
\]

The map $\varphi$ induces a natural map $\bw^n_R (\varphi)\col
\bw^n_R(N)\to \bw^n_R(R^n)$. If $a\in R$ is in the image of this map,
via the identification of $R$ with $\bw^n_R(R^n)$, then $aM=0$.

In particular, if $M$ is faithful, then $\bw^n_R(\varphi)=0$.
\end{chunk}

\begin{proposition}
\label{tor2}
  Consider a short exact sequence of $R$-modules
\[
0\lra N\xrightarrow{\ \varphi\ } R^2\xrightarrow {\ \psi\ } M\lra 0
\]

If $\Tor_2^R(M,M)=0$ and $M$ is faithful, then $\nu(N)\le 1$.
\end{proposition}

\begin{proof} For any module $L$ we define a map $\iota^L\col
\bw^2_R(L)\to L\otimes_RL$ given by
\[
\iota^L(x\wedge y)=x\otimes y-y\otimes x
\]

The hypothesis implies $\Tor_1^R(M,N)=0$. It follows that the induced
map
\[
\varphi\otimes_R N\col N\otimes_RN\to R^2\otimes_RN
\]
is injective.  The map $\varphi\otimes_R\varphi\col N\otimes_RN\to
R^2\otimes_RR^2$ is the composition
\[
N\otimes_RN\xrightarrow{\varphi\otimes_R
N}R^2\otimes_RN\xrightarrow{R^2\otimes_R\varphi}R^2\otimes_RR^2
\]
Both maps are injective, hence so is $\varphi\otimes_R\varphi$.

Recall from \ref{wedge} that $\bw^2_R(\varphi)=0$. The commutative
diagram
\[
\xymatrixrowsep{2.2pc} \xymatrixcolsep{3.4pc}
\xymatrix{
\bw^2_R(N)\ar@{->}[r]^{\iota^N}\ar@{->}[d]_{\bw^2_R(\varphi)}&N\otimes_RN\ar@{^{(}->}[d]^{\varphi\otimes_R\varphi}\\
\bw^2_R(R^2)\ar@{->}[r]^{\iota^R}&R^2\otimes_RR^2
}
\]
then yields $\iota^N=0$. In particular, the map $\iota^N\otimes_Rk$ is
zero.  Note that this map can be identified with
$\iota^{N\otimes_Rk}$. As $N\otimes_Rk$ is a finite direct sum of
copies of $k$, the map $\iota^{N\otimes_Rk}$ is clearly injective. It
follows that $\bw^2_R(N\otimes_Rk)=0$, hence $N\otimes_Rk\cong
k$. Nakayama's Lemma then shows that $N$ is cyclic.
 \end{proof}

\begin{proof}[Proof of Theorem {\rm \ref{tor22}}]
  (1) Assume that $R$ is not Gorenstein, hence $\type(R)=2$. Set
  $N=\omega_1$. We denote $e(L)$ the multiplicity of an $R$-module
  $L$. Since $e(\omega)=e(R)$, we use the fact that multiplicity is
  additive on short exact sequences of maximal Cohen-Macaulay modules
  to obtain $e(N)=e(R)$. By Proposition \ref{tor2}, $N$ is cyclic,
  hence there is a surjection $R\to N$. If $K$ is the kernel of this
  map, then $e(K)=0$, hence $N$ is free, a contradiction.

(2) By \cite[B.4]{ABS} we have $\Tor_i^R(\omega,\omega)=0$ for $i=1,2$
    so we can apply (1).
\end{proof}

\end{document}